\documentclass[10pt]{article}
\usepackage{amsmath,amssymb}
\usepackage{algorithmic}
\usepackage{algorithm}
\usepackage{graphicx}
\usepackage{arydshln}
\usepackage{booktabs}
\usepackage{chngcntr}
\counterwithout{figure}{section}
\counterwithout{table}{section}

\newtheorem{thm}{Theorem}[section]
\newtheorem{corollary}[thm]{Corollary}
\newtheorem{lemma}[thm]{Lemma}

\newtheorem{assumption}[thm]{Assumption}

\newtheorem{proposition}[thm]{Proposition}

\newtheorem{assu1}[thm]{Assumption}

\newtheorem{definition}[thm]{Definition}

\newtheorem{example}[thm]{Example}

\newtheorem{remark}[thm]{Remark}

\newtheorem{protocol}[thm]{Protocol}

\newcommand{\tL}{\mathcal{L}}

\newcommand{\N}{\mathbb{N}}
\newcommand{\R}{\mathbb{R}}

\newcommand{\Z}{\mathbb{Z}}

\newcommand{\spa}{\mathrm{span}\,}

\newcommand{\vol}{\mathrm{vol}\,}

\newcommand{\norm}[1]{\lVert #1 \rVert}
\newcommand{\argmax}{\mathrm{argmax}\,}

\newcommand{\openbox}{\leavevmode
  \hbox to.77778em{%
    \hfil\vrule
  \vbox to.675em{\hrule width.6em\vfil\hrule}%
  \vrule\hfil}}
\newcommand{\proofname}{Proof}

\newenvironment{proof}[1][\proofname]{\par\normalfont
  \trivlist\item[\hskip\labelsep\itshape #1:]\ignorespaces
  }{\hspace*{1cm}\hspace*{\fill}\openbox \medskip\endtrivlist}

\title{Improvements in closest point search based on dual HKZ-bases.
  \footnote{Partially supported by ArmaSuisse funding ARAMIS R3210/047-12 and SNF grant No. 121874.} \\
  {\small \textbf{}}
}

\date{\today}
\author{Urs Wagner and G\'erard Maze\\
{\small {\em e-mail:\/} \{urs.wagner,gmaze\}@math.uzh.ch \vspace{-1mm} }\\
{\small Mathematics Institute\vspace{-1mm}}\\
{\small University of Z\"urich\vspace{-1mm}}\\
{\small Winterthurerstr 190, CH-8057 Z\"urich, Switzerland }
\vspace{3mm} }

\begin{document}
\maketitle

\begin{abstract}
In this paper we review the technique to solve the CVP based on dual HKZ-bases by J. Bl\"omer \cite{bl00}. The technique is based on the transference
theorems given by Banaszczyk \cite{ba93} which imply some necessary conditions on the coefficients of the closest vectors with respect to a basis whose dual
is HKZ reduced. Recursively, starting with the last coefficient, intervals of length $i$ can be derived for the $i$th coefficient of any closest vector.
This leads to $n!$ candidates for closest vectors. 
In this paper we refine the necessary conditions derived from the transference theorems, giving an exponential reduction of the number of candidates. 
The improvement is due to the fact that the lengths of the intervals are not independent. In the original algorithm the candidates for 
a coefficient pair $(a_i,a_{i+1})$ correspond to the integer points in a rectangle of volume $i \cdot (i+1)$.  
In our analysis we show that the candidates for $(a_i,a_{i+1})$ in fact lie in an ellipse with transverse and conjugate diameter $i+1$, respectively $i$. 
This reduces the overall number of points to be enumerated by an exponential factor of about $0.886^n$.
We further show how a choice of the coefficients $(a_n,\dots,a_{i+1})$ influences the interval from which $a_i$ can be chosen.
Numerical computations show that these considerations allow to bound the number of points to be enumerated by
$n^{0.75 n}$ for $10 \leq n \leq 2000$.
Under the assumption that the Gaussian heuristic for the length of the shortest nonzero vector in a lattice is tight, 
this number can even be bounded by $\frac{1}{2^{2n}} n^{n/2}$. 
\end{abstract}

\vspace{3mm}
\vspace{3mm}

\section{Introduction}

The closest vector problem (CVP) is the problem of finding a closest lattice point of a given
lattice $\tL \subset \R^n$ to an arbitrary point $t$ in $\R^n$. While the problem is proven to be 
NP-hard (see e.g. \cite{mi02}), algorithms exist to solve the problem approximately in polynomial time. 
 Babai's nearest plane algorithm \cite{ba86} is the generic way to get an approximate solution,
and the quality of the solution substantially depends on the quality of the basis it is applied to.
The algorithm recursively selects the nearest $n-1, n-2,\dots, 0$ dimensional plane spanned by the basis vectors.
 The more orthogonal the basis vectors are, the 
better the output of the algorithm. E.g. if the basis is LLL-reduced, the point found lies within $2(4/3)^{n/2}$  
times the distance of a closest lattice point
to $t$ \cite{mi02}. If the basis vectors are even pairwise orthogonal (note that such a basis does not necessarily exist), it returns
a closest vector.
Babai's nearest plane algorithm can be modified to output an exact solution by not only considering the nearest, but all planes 
with distance up to 
a certain bound in the recursion steps. This is exactly the approach of Kannan \cite{ka87}. Note that once a plane is fixed, the problem translates 
to finding a closest lattice point in a lower dimensional lattice, namely the orthogonal projection of the lattice onto that plane.
 Clearly the number of planes to be considered in the lower dimensional lattice is dependent on
 the choice
of the plane in the upper dimension which was realized by Pohst \cite{fi85}. So instead of looking at all points inside a parallelepiped, 
the points inside a hyperellipsoid are considered.
The running time of Kannan's and Pohst's approach was proven to be $O(n^n)$ in the original considerations \cite{ka87}.
 Recently, refined analysis by Hanrot and Stehl\'e showed that applying Kannan's algorithm to a HKZ-basis the closest vectors 
can be found by enumerating $2^{O(n)}n^{0.5n}$ points.
 A more elaborate survey on different methods to solve the problem exactly can be found in e.g. \cite{ag02}. 
In \cite{bl00}, a different approach than the one of Kannan \cite{ka87} is presented. The main difference is that the basis used for closest point
search is dual HKZ reduced, e.g. it is a basis whose dual is HKZ reduced. Due to the special form of the basis, the transference theorems proven by 
Banaszczyk \cite{ba93} can be used to bound the number of planes to be considered.
In each recursion step the number of planes to be considered decreases by $1$. Having $n$ planes to consider in the first recursion step, 
this results in enumeration of $n!$ lattice points. Recently Micciancio gave an algorithm to solve the CVP in time $2^{O(n)}$ based on 
Voronoi cell computations \cite{mi10}.
The caveat in this approach is the exponential space requirement, and it is not (yet) clear how this can be reduced.

In this paper we give a refined analysis of the approach given in \cite{bl00}. We show how the overall number of points to be enumerated can be decreased. 
While in the original algorithm the number of choices of the planes is bounded independently in each step, we examine
 how the choice of a plane in early recursion steps 
influences the possible number of choices in following steps. We show how to decrease the number of lattice points to
 be enumerated by an exponential factor $(\pi/4)^{n/2}$ by deriving how the choices of the planes in 
two neighboring recursion steps are connected. Further we derive a recursive formula (in 
the dimension of the lattice) for the number of points to be enumerated when the 
choices made in early recursion steps are rigorously used to constrain the further choices. A closed form approximation of this formula is still an open
problem. However numerical computations show that this number can be bounded by $n^{0.75n}$
for $10< n \leq 2000$. Given that the shortest vector
of the dual lattice satisfies the Gaussian heuristic, we show that this number can even be bounded by $\frac{1}{2^{2n}}n^{n/2}$.

The paper is organized as follows.
In Section \ref{sec:back} we give some background and introduce notation used throughout the paper. 
In Section \ref{sec:old} the original algorithm
proposed in \cite{bl00} is described and a motivation for further studies of it is given. 
In Section \ref{sec:new} we show how the running
time can be sped up by a factor $(\pi/4) ^{n/2}$. 
In Section \ref{sec:further}, a recursive formula bounding the number of points to be enumerated is
derived and its behavior is analyzed. 
Section \ref{sec:Hermite} shows how the bound on the number of points can even be further reduced under the assumption that
the Gaussian Heuristic is tight. Finally in Section \ref{sec:Kannan}, Kannan's algorithm and the analysis by Hanrot and Stehl\'e
is quickly reviewed.
Some concluding remarks are given in Section \ref{sec:Conclusion}.

\section{Background and Notation}
\label{sec:back}

Throughout the paper $\norm{\cdot}$ denotes the euclidean norm.
Let $\tL$ be the discrete subgroup generated by integer linear combinations of $k$ linearly independent vectors $b_1,\dots,b_k$ in $\R^n$. 
We call $\tL$ a lattice of rank $k$ 
and dimension $n$.
Given a lattice basis $\{b_1,\dots,b_n\}$ of $\tL$ we will usually write it as rows of a matrix $B$ in the following way
\[
B=[b_1,\dots,b_n].
\]
The lattice points in $\tL$ are the integer linear combinations of the basis vectors,
\[
\tL=\tL(B)=\{xB  |  x \in \Z^n\}. 
\]

By $B^*=[b^*_1,\dots,b^*_n]$ we denote the usual Gram-Schmidt basis corresponding to $B=[b_1,\dots,b_n]$.
And by $\pi_i$  we denote the orthogonal projection
\[
\pi_i: \spa(b_1,\dots,b_n) \longrightarrow \spa(b_1,\dots,b_{i-1})^{\bot}
\]
Further with $\tL:=\tL(b_1,\dots,b_n)$ we have that 
\[
 \tL_i:=\pi_i(\tL)
\]
is again a lattice of rank $n-i+1$ with basis $\{ \pi_i(b_i),\dots,\pi_n(b_n) \}$.

\begin{definition}
 Given a lattice $\tL \subset \R^n$ the dual lattice $\tL^\times$ is defined by
\[
\tL^\times :=\{v \in \R^n : \langle v,w \rangle \in \Z \mbox{ } \forall w\in \tL\}
\]
\end{definition}
There exist a unique dual basis $B^\times$ for every basis $B$ of $\tL$.

\begin{definition}
 Given a basis $B=[b_1,\dots,b_k]$ for a lattice $\tL \subset \R^n$ of rank $k$, then $B^\times=[b_1^\times,\dots,b_k^\times]$ is the reverse dual basis 
if and only if
\[
b_i^\times \in \spa(b_1,\dots,b_k) \mbox{ and } \langle b_i^\times,b_j \rangle = \delta_{i,k-j+1}
\]
\end{definition}
From now on we will assume that the lattice has full rank, i.e. $k=n$.

\begin{remark}
Given $v=v_1b_1+\dots+v_nb_n \in \tL$, then 
\[
 v_i= \langle v, b_{n-i+1}^\times \rangle
\]
is the $i$-th coordinate of $v$ with respect to the basis $B$.
\end{remark}
The algorithm of this paper uses lattice bases of special form.
\begin{definition}
A basis $B=[b_1,\dots,b_n]$ of a lattice $\tL(b_1,\dots,b_n)$ is called HKZ-basis if and only if it satisfies the following two conditions
\begin{enumerate}
 \item $\frac{\langle b_i,b^*_j \rangle}{\langle b^*_j,b^*_j \rangle} \leq \frac{1}{2}$ for $j<i$ (size-reduced).
 \item The $i$-th Gram-Schmidt vector satisfies $|b_i^*| = \lambda_1(\pi_i(\tL))$. 
\end{enumerate}
\end{definition}
\begin{definition}
A basis $B=[b_1,\dots,b_n]$ of a lattice $\tL(b_1,\dots,b_n)$ is called dual HKZ-basis when its reverse-dual basis 
$B^\times=[b_1^\times,\dots,b_n^\times]$
is HKZ-reduced.
\end{definition}
\begin{lemma}
\label{lem:dualprop}
  Let $B=[b_1,\dots,b_n]$ a basis with dual basis $B^\times=[b_1^\times,\dots,b_n^\times]$. Then the dual basis
of $[b_1,\dots,b_{n-j}]$ equals $[\pi_{j+1}(b^\times_{j+1}),\dots,\pi_{j+1}(b^\times_{n})]$, $n-1\geq j\geq 0$.
\end{lemma}
\begin{proof}
 We start by showing that $\pi_{j+1}(b^\times_{i}) \in \spa(b_1,\dots,b_{n-j})$ for $i \geq j+1$. Clearly 
$\pi_{j+1}(b^\times_{i}) \in \spa(b_1,\dots,b_{n})$. Also $\pi_{j+1}(b^\times_{i}) \in \spa(b^\times_1,\dots,b_j^\times)^{\bot}$.
Hence 
\[
\pi_{j+1}(b^\times_{i}) \in \spa(b_1,\dots,b_{n}) \cap \spa(b^\times_1,\dots,b_j^\times)^{\bot}=\spa(b_1,\dots,b_{n-j}).
\]
It remains to show that $\langle \pi_{j+1}(b^\times_{i}),b_{k}\rangle=1$ if $k=n+1-i$ and $\langle \pi_{j+1}(b^\times_{i}),b_{k}\rangle=0$ if 
$k \in \{1,\dots,n-j\} \backslash \{n+1-i\}$.
This is straightforward as with $k<n-j+1$
\[
 \langle \pi_{j+1}(b^\times_{i}),b_k \rangle=\langle b^\times_{i},b_k \rangle.
\]
\end{proof}
This proves Lemma 1 in \cite{bl00}:
\begin{lemma}
 If $[b_1,\dots,b_n]$ is a dual HKZ-basis for $\tL(b_1,\dots,b_n)$ then $[b_1,\dots,b_k]$ is a dual HKZ-basis for $\tL(b_1,\dots,b_k)$, $k=1, \dots, n$.
\label{lem:subdual}
\end{lemma}
Clearly $b^*_k \in \spa(b_1,\dots,b_k)$, and as $\langle b_i,b^*_k\rangle=0$ for $i<k$ it follows that
\[
 \langle b^*_k,b_k \rangle =\langle b^*_k,b^*_k \rangle = |b^*_k|^2.
\]
So we have that $\frac{b^*_k}{|b^*_k|^2}$ is the first basis vector of the basis dual to $[b_1,\dots,b_k]$. 
Hence we get the following corollary.
\begin{corollary}
$\frac{b^*_k}{|b^*_k|^2}$ is a shortest vector in $\tL^\times(b_1,\dots,b_k)$ and $\frac{1}{|b^*_k|}=|\lambda_1(\tL^\times(b_1,\dots,b_k))|$.
Further in a dual HKZ-reduced basis $B=[b_1,\dots,b_n]$, $\norm{b^*_k}$ is maximal under all possible bases for the sublattice $\tL(b_1,\dots,b_k)$. 
\label{cor:subgram}
\end{corollary}
We will now state a theorem from the geometry of numbers by Banaszczyk \cite{ba93}. First we need two definitions.
 Let $\mu(\tL)$ denote the covering radius of a lattice, i.e.
\[
\mu(\tL) := \max_{t \in \spa(\tL)} \min_{x \in \tL} \norm{x-t}.
\]
Denote the set of all $i$-tuples of linearly independent lattice vectors as $V_i$. Then the $i$-th minimum $\lambda_i(\tL)$ of a lattice is defined as
\[
 \lambda_i(\tL) := \min_{(v_1,\dots,v_i) \in V_i} \max_{1\leq j \leq i} \norm{v_j}.
\]
\begin{thm}[Transference Theorems] 
The successive minimas $\lambda_i(\tL)$ and covering radius $\mu(\tL)$ of a lattice $\tL$ of rank $n$ satisfy the following bounds
\begin{enumerate}
 \item $\lambda_i(\tL)\cdot \lambda_{n-i+1}(\tL^\times)\leq n$, $i=1,\dots, n$,
 \item $\mu(\tL) \cdot \lambda_1(\tL^\times) \leq \frac{n}{2}$.
\end{enumerate}
\end{thm}
With Corollary \ref{cor:subgram}  we have that
$\mu(\tL) \cdot \lambda_1(\tL^\times) = \frac{\mu(\tL)}{\norm{b^*_n}} $ so the second inequality in the Transference Theorems implies that 
\begin{equation}
 \mu(\tL) \leq \frac{n}{2} \norm{b^*_n}.
\label{equ:mu}
\end{equation}

\section{Original approach}
\label{sec:old}

In this section we review the approach presented in \cite{bl00}. Given a lattice $\tL=\tL(b_1,\dots,b_n)$ in $\R^n$ and a vector $t\in \R^n$, we want 
to find a vector $v$ such that 
$\norm{v-t}\leq \norm{w-t}$ for all $w\in \tL$. We assume that the basis $B=[b_1,\dots,b_n]$ is dual HKZ reduced.
\begin{enumerate}
  \item $e=e_1b^*_1+\dots+e_nb^*_n=t-v$ denotes the error vector,
  \item $e^{(i)}:=e-\sum^{n}_{j=i+1} e_{j}b^*_{j}$ is the orthogonal projection of the error vector onto $\spa(b_1,\dots,b_i)$.
  \item $\mu^{(i)}$ denotes the covering radius of $\tL(b_1,\dots,b_{i})$,
  \item $\lambda^{\times (i)}_1:=\lambda_1(\tL^{\times}(b_1,\dots,b_{i}))$.
\end{enumerate}
So suppose $v=c_1b_1+\dots+c_nb_n$, $c_i \in \Z$ is a closest vector to $t=t_1b_1+\dots+t_nb_n$, $t_i \in \R$. With (\ref{equ:mu}) we get
\[
 \norm{v-t} \leq \mu(\tL) \leq \frac{n}{2} \norm{b^*_n},
\]
and as $(c_n-t_n)^2 \norm{b^*_n}^2 \leq  \norm{v-t}^2 \leq \left(\frac{n}{2}\right)^2 \norm{b^*_n}^2$ we have
\[
 |c_n-t_n| \leq \frac{n}{2}.
\]
Hence we get an interval of length $n$ for the $n$-th coordinate $c_n$ of $v$:
\begin{equation}
\label{equ:interval}
c_n \in [t_n-n/2,t_n+n/2].
\end{equation}
As $c_n \in \Z$ we can enumerate $n$ values for $c_n$.
Note that for the orthogonal projection $t^{(n-1)}$ of $t-c_n  b_n$ onto $\spa(b_1,\dots,b_{n-1})$ we have
\[
t^{(n-1)}= t-c_n  b_n-\frac{\langle t-c_nb_n,b_n^*\rangle}{\langle b_n^*,b_n^*\rangle} b^*_n=t-c_n  b_n-(t_n-c_n)b_n^*,
\]
and hence $(t_n-c_n)=e_n$.
The following lemma \cite{bl00} allows to recursively carry the problem to proper sublattices of $\tL$ in order to derive corresponding bounds for the 
other coordinates of $v$.
\begin{lemma}
A vector $w \in \tL(b_1,\dots,b_{i})$ is a closest vector to $t-\sum_{j > i} x_j b_j$, $x_j\in \Z$ if and only if $w$ is a closest vector of the 
orthogonal projection $t^{(i)}$ of $t-\sum_{j > i} x_j b_j$ onto $\spa(b_1,\dots,b_{i})$.
\label{lem:indstep}
\end{lemma}
So given  $c_{i+1},\dots,c_n$ and $e_{i+1},\dots,e_n$ the problem reduces to finding the closest vector
to $t^{(i)}=t-\sum_{j=i+1,\dots, n} c_{j}b_{j} -\sum_{j=i+1,\dots,n} e_{j}b^*_{j}$ in the lattice $\tL(b_1,\dots,b_i)$ of rank $i$.
As by Lemma \ref{lem:subdual} $[b_1,\dots,b_i]$ is a dual HKZ basis for $\tL(b_1,\dots,b_i)$, we can recursively take the problem to a lower dimension.
In dimension $i=1$, $t^{(1)} \in \spa(b_1)$ and we set
$c_1 =\lfloor \frac{\langle t^{(1)}, b_1\rangle}{\langle b_1, b_1 \rangle}  \rceil$ 
and 
$e_1=\frac{\langle t^{(1)}, b_1\rangle}{\langle b_1, b_1 \rangle}- \lfloor \frac{\langle t^{(1)}, b_1\rangle}{\langle b_1, b_1 \rangle}  \rceil$ 
in order to get the closest lattice vector in $\tL(b_1)$ to $t^{(1)}$. In fact
\[
t^{(1)}-c_1b_1-e_1b^*_1= t-\sum^n_{j=1} c_jb_j -\sum^n_{j=1} e_jb^*_j =0,
\]
assuring that we get a valid pair of vectors $v \in \tL$ and error $e \in \R^n$ in the sense that $v+e=t$.
Hence we have the following lemma:
\begin{lemma}
Recursively we can derive $n!$ candidates for a closest vector to $t$ in $\tL$ given a dual HKZ-basis for $\tL$.
\end{lemma}
We will now give a short motivation for further analysis.
The algorithm and the corresponding bound is not optimized at all. Suppose the $n$-th coordinate $e_n$ of the error vector equals $\frac{n}{2}$.
Clearly we have the following inequality $\frac{n}{2} = \frac{\langle e,b^*_n \rangle}{\langle b_n^*,b^*_n \rangle}=\norm{e}\frac{1}{\norm{b_n^*}} \cos \gamma.$
As $\norm{e}\leq \mu(\tL)$, with Equation (\ref{equ:mu}) we get $\frac{n}{2} \leq \frac{n}{2} \cos \gamma$. Consequently $\gamma = 0$ which means that the error vector points 
exactly in the direction of $b^*_n$. So the error vector can be written as multiple of $b^*_n$ and the coefficients
$e_1,\dots,e_{n-1}$ are trivially zero. In the next section we will see how the value of $e_i$ influences 
the interval length in which $e_{i-1}$ lies.

\section{First Improvement}
\label{sec:new}
Given the same problem and notation as in Section \ref{sec:old}, let us consider the following set
\[
 T_n:=\left\{(e_1,\dots,e_n)\in \R^n: v=t-\sum^n_{j=1} e_j b^*_j \in \tL \mbox{ and } |e_j|\leq \frac{j}{2} \right\}.
\]
In the last section we have seen how all elements of this set can be enumerated recursively and that due to the dual HKZ reducedness of $B$ 
in fact all closest vectors to $t$ are in the set
$\{t-\sum^n_{j=1} e_j b^*_j: (e_1,\dots,e_n) \in T_n \}$.
Further in each recursion step the value $e_i+c_i$ is given and as $c_i$ is an integer, the condition $|e_i|\leq \frac{i}{2}$ 
implies $i$ possible values for $e_i$.
So $|T_n|$ is upper bounded by $n!$. 

The goal of this section is to define a subset $T'_n \subset T_n$ still having the property that  all closest vectors to $t$ are in the set
$\{t-\sum^n_{j=1} e_j b^*_j: (e_1,\dots,e_n) \in T'_n \}$. 
We will now show how additional constraints on the $e_i$'s can be derived.
Recall that the condition $|e_i|\leq \frac{i}{2}$ comes from the fact that $\norm{e^{(i)}} \leq \mu^{(i)} \leq \frac{i}{2} \norm{b^*_i}$, 
where the second inequality is due to the dual HKZ reducedness of the basis.
This implies $\norm{e_ib^*_i} \leq \norm{e^{(i)}} \leq \mu^{(i)} \leq \frac{i}{2} \norm{b^*_i}$ and consequently $|e_i| \leq \frac{i}{2}$.
However $\norm{e^{(i)}} \leq \mu^{(i)}$ is not the only bound on $\norm{e^{(i)}}$ we have. Clearly also
\[
 \norm{e^{(i)}}^2 = \norm{e^{(k)}}^2 - \sum^k_{j=i+1} e^2_j\norm{b^*_j}^2 \leq \mu^{(k)2} - \sum^k_{j=i+1} e^2_j\norm{b^*_j}^2 \mbox{ for all } k\geq i. 
\]
Now if $e^2_j > \frac{1}{4}$, $j=i+1,\dots,k$ with $\mu^{(k)2} \leq \mu^{(i)2} + \frac{1}{4} \sum^k_{j=i+1} \norm{b^*_j}^2$ we have 
a tighter upper bound
\begin{equation}
 \norm{e^{(i)}}^2 \leq \mu^{(k)2} - \sum^k_{j=i+1} e^2_j\norm{b^*_j}^2  < \mu^{(i)2}.
\label{equ:tighter} 
\end{equation}
This observation can now be exploited to reduce the size of the intervals in which the $e_i$'s lie.
For all $i=2,\dots,n$, we derive factors $A_i(e_i) \in \R$ depending on $e_i$, such that $\norm{e^{(i-1)}} \leq A_i(e_i) \mu^{(i-1)}$
 and consequently $|e_{i-1}| \leq A_i(e_i) \frac{i-1}{2}$. Let us define
\begin{equation}
 A_i^2(c):=\frac{\frac{i^2}{4}-c^2}{\frac{i^2}{4}-\frac{1}{4}} \mbox{, } i\in \N .
\end{equation}
We obtain the following lemma
\begin{lemma}
If $c^2 \geq \frac{1}{4}$, then we have
\[
 \mu^{(i)2}-c^2 \norm{b^*_i}^2 \leq A^2_i(c) \cdot \mu^{(i-1)2}.
\]
\label{lem:A}
\end{lemma}
\begin{proof}
 We have to show that 
\[
\left(\frac{i^2}{4}-\frac{1}{4}\right) \left(\mu^{(i)2}-c^2 \norm{b^*_i}^2\right)\leq  \left(\frac{i^2}{4}-c^2\right) \mu^{(i-1)2}.
\]
Since $ \mu^{(i)2}- \frac{1}{4} \norm{b^*_i}^2 \leq  \mu^{(i-1)2}, $
it is sufficient to show that
\[
\left(\frac{i^2}{4}-\frac{1}{4}\right) \left(\mu^{(i)2}-c^2 \norm{b^*_i}^2\right)\leq  \left(\frac{i^2}{4}-c^2\right)\left(\mu^{(i)2}- \frac{1}{4} \norm{b^*_i}^2\right) .
\]
This is true since
\[
 \left(c^2-\frac{1}{4}\right)\mu^{(i)2}  \leq \left(c^2 - \frac{1}{4}\right)\frac{i^2}{4}\norm{b^{*}_i}^2.
\]
\end{proof}

We can now prove the core lemma, which gives the factor by which the error vector is smaller than the covering radius.

\begin{lemma} 
Under the previous assumptions and notations:
\begin{equation} 
\norm{e^{(i-1)}}^2  \leq A^2_i(e_i) \cdot \mu^{(i-1)2}.
\end{equation}
\label{prop:anker}
\end{lemma}
\begin{proof}
We separate the two cases where $|e_i|<\frac{1}{2}$, $|e_i|\geq \frac{1}{2}$ respectively.
If $|e_i|<\frac{1}{2}$, then $A^2_i(e_i)>1$ and the proposition follows by $\norm{e^{(i-1)}}^2 \leq \mu^{(i-1)2}$.
If $|e_i|\geq \frac{1}{2}$, the claim follows from
\[
 \norm{e^{(i-1)}}^2 = \norm{e^{(i)}}^2-e^2_i \norm{b^*_i}^2 \leq \mu^{(i)2}-e^2_i \norm{b^*_i}^2,
\]
and Lemma \ref{lem:A}.
\end{proof}
So with $e^2_{i-1}\norm{b^*_{i-1}}^2 \leq \norm{e^{(i-1)}}^2\leq  A^2_i(e_i) \cdot \mu^{(i-1)2}$ and $\frac{\mu^{(i-1)2}}{\norm{b^*_{i-1}}^2} \leq \frac{(i-1)^2}{2} $
we immediately obtain the following bound.
\begin{corollary}
Using the notation from before, 
\begin{equation}
\label{equ:pairs}
 e^2_{i-1} \left(\frac{i^2}{4}-\frac{1}{4}\right)+e^2_i\frac{(i-1)^2}{4}  \leq  \frac{i^2}{4}\frac{(i-1)^2}{4}.
\end{equation}
\end{corollary}
So we define 
\[
T'_n:=\{(e_1,\dots,e_n)\in T_n: \mbox{Equation } (\ref{equ:pairs}) \mbox{ holds for all } i=2,\dots,n \}.
\]
We are interested on an upper bound on the volume of $T'_n$ giving us an upper bound on the number of points we have to enumerate to get the closest vectors.
Let us first assume that $n$ is even. Clearly
\[
T'_n \subset T''_n := \bigotimes^{n/2}_{i=1} \{(e_{2i-1},e_{2i})\in \R^n: \mbox{Equation } (\ref{equ:pairs}) \mbox{ holds } \}.
\]
The volume of $T''_n$ can be computed as the product of the volumes of the $2$-dimensional ellipses. 
\[
 \vol(T''_n)=  \left(\frac{\pi}{4}\right)^{ n/2} \prod^{n/2}_{i=1} \left(2i (2i-1) \frac{i}{\sqrt{i^2-1/4}}\right) =\left(\frac{\pi}{4}\right)^{ n/2 } n!  \prod^{n/2}_{i=1}  \frac{i}{\sqrt{i^2-1/4}}.
\]
As
\[
\prod^{n/2}_{i=1}  \frac{i}{\sqrt{i^2-1/4}}= \sqrt{\frac{4}{3}}\prod^{n/2}_{i=2}  \frac{i}{\sqrt{i^2-1/4}}
<\sqrt{\frac{4}{3}}\prod^{n/2}_{i=2}  \frac{i}{\sqrt{i^2-1}}=\sqrt{\frac{4}{3}} \underbrace{\sqrt{\frac{n}{n/2+1}}}_{<\sqrt{2}} 
<2.
\]
We obtain
\[
 \vol(T''_n)< 2\left(\frac{\pi}{4}\right)^{ n/2 } n! .
\]
In the case where $n$ is odd consider 
\[
T'_n \subset T'''_n :=  \{|e_1|\leq \frac{1}{2}\}  \otimes \bigotimes^{(n-1)/2}_{i=1} \{(e_{2i},e_{2i+1})\in \R^n: \mbox{Equation } (\ref{equ:pairs}) \mbox{ holds } \}.
\]
The volume of $T'''_n$ then is
\[
 \vol(T'''_n)=  \left(\frac{\pi}{4}\right)^{ (n-1)/2} \prod^{(n-1)/2}_{i=1} 2i (2i+1) \frac{2i+1}{\sqrt{(2i+1)^2-1}} = \left(\frac{\pi}{4}\right)^{ (n-1)/2} n! \prod^{(n-1)/2}_{i=1} \frac{i+1/2}{\sqrt{(i+1/2)^2-1/4}}.
\]
As  $\prod^{(n-1)/2}_{i=1} \frac{i+1/2}{\sqrt{(i+1/2)^2-1/4}}< \prod^{(n-1)/2}_{i=1} \frac{i}{\sqrt{i^2-1/4}}$ we get the same bound
\[
 \vol(T'''_n)< 2 \left(\frac{\pi}{4}\right)^{ n/2 } n! .
\]

\begin{thm}
 Given a dual HKZ basis $B$ of a full rank lattice $\tL \subset \R^n$ all closest vectors to a given point $t\in \R^n$ can be found
by recursively enumerating at most $2 \left(\frac{\pi}{4}\right)^{ n/2 } n!$ lattice points.
\end{thm}
So with $\sqrt{\pi/4} \approx 0.886$ we get an exponential gain of roughly $0.886^n$ compared to the original considerations.

\section{Further improvement} \label{sec:further}

Recall the starting point of the considerations of the previous section. We have an upper bound on $\norm{e^{(i)}}^2$:
\begin{equation}
 \norm{e^{(i)}}^2 \leq \mu^{(k)2} - \sum^k_{j=i+1} e^2_j\norm{b^*_j}^2. 
\label{equ:bound}
\end{equation}
Note that the bound (\ref{equ:bound}) is decreasing with increasing $e_j$'s and in fact if they satisfy $|e_j|>\frac{1}{2}$ then 
as in Equation (\ref{equ:tighter}),
\[
 \mu^{(k)2} - \sum^k_{j=i+1} e^2_j\norm{b^*_j}^2 < \mu^{(i)2}.
\]
In the original approach (see Section \ref{sec:old}), only the case $k=i$ was considered. In Section \ref{sec:new} we considered
 the case where $k=i+1$ and we got that  
\[
\norm{e^{(i)}}^2  \leq A^2_{i+1}(e_{i+1}) \cdot \mu^{(i)2},
\]
where $A_{i+1}^2(e_{i+1}):=\left(\frac{(i+1)^2}{4}-e_{i+1}^2\right)\left(\frac{(i+1)^2}{4}-\frac{1}{4}\right)^{-1}$. From that we derived that pairs of coefficients $(e_i,e_{i+1})$
lie inside a $2$-dimensional ellipsoid of volume $\frac{\pi}{4}(i+1)^2\sqrt{\frac{i}{i+2}}$.
The goal of this section is to generalize this method to more than just 2-tuples of coefficients.
Consider 
\begin{eqnarray}
\norm{e^{(i-1)}}^2  = \norm{e^{(i)}}^2 - e^2_{i}\norm{b^*_{i}}^2  & \leq&  
  A^2_{i+1}(e_{i+1}) \cdot \mu^{(i)2} - e^2_{i}\norm{b^*_{i}}^2 \nonumber
\\ &=& A^2_{i+1}(e_{i+1})\left(\mu^{(i)2}- \frac{e^2_{i}}{A^2_{i+1}(e_{i+1})}\norm{b^*_{i}}^2 \right) .
\end{eqnarray}
So under the condition that $\frac{e^2_{i}}{A^2_{i+1}(e_{i+1})} \geq \frac{1}{4}$, by Lemma \ref{lem:A} we have
\[
\norm{e^{(i-1)}}^2 \leq A^2_{i+1}(e_{i+1})  A^2_{i}\left( \frac{e^2_{i}}{A^2_{i+1}(e_{i+1})}\right) \mu^{(i-1)2}. 
\]
Note that if $|e_{i+1}|, |e_i| > \frac{1}{2}$ , $A^2_{i+1}(e_{i+1})<1$ and $A^2_{i}\left( \frac{e^2_{i}}{A^2_{i+1}(e_{i+1})}\right)<1$.
Clearly the bigger $|e_{i+1}|, |e_i|$ the smaller the bound on $\norm{e^{(i-1)}}$ becomes.

\begin{definition}
For $e_n,\dots,e_1$ recursively define $C^2_{n+1},\dots,C^2_1$ by
\[
C^2_{n+1}:=1 
 \mbox{ \ and \ }
C^2_{i-1}:=\left\{\begin{array}{ll} 1 & \mbox{if } |e_{i-1}| < \frac{1}{2}, \\ 
			C_{i}^2 A^2_{i-1}\left(\frac{e_{i-1}}{C_{i}}\right) & \mbox{else.} \end{array}\right. 
\]
\end{definition}
Note that $C^2_i\leq 1$ for all $i$.
\begin{proposition}
\label{prop:eC}
For $i=n,\dots,2$ we have
\[
\norm{e^{(i-1)}}^2 \leq C^2_{i} \mu^{(i-1)2}.
\]
\end{proposition}
\begin{proof}
The proof goes by reverse induction on $i$.
For $i=n-1$ the result follows by Proposition \ref{prop:anker}. Assume the results holds for $i$.
If $|e_i| < \frac{1}{2}$, $C^2_i=1$ and the proposition follows trivially.
For the case $|e_i|\geq \frac{1}{2}$, note that
\[
 \norm{e^{(i-1)}}^2=\norm{e^{(i)}}^2-e^2_{i}\norm{b^*_i}^2 \leq C^2_{i+1}\mu^{(i)2} - e^2_i\norm{b^*_i}^2= C^2_{i+1}\mu^{(i)2} - e^2_i\norm{b^*_i}^2.
\]
We also have that  $\frac{e^2_i}{C^2_{i+1}}>\frac{1}{4}$
and with Lemma \ref{lem:A},
\[
 C^2_{i+1}\left(\mu^{(i)2} - \frac{e^2_i}{C^2_{i+1}}|b^*_i|^2\right) \leq C^2_{i+1}A^2_{i}\left(\frac{e_i}{C_{i+1}}\right) \mu^{(i-1)2}=C^2_{i} \mu^{(i-1)2} .
\]
\end{proof}
With the Transference Theorems the following corollary follows immediately:
\begin{corollary}
 For $i=n,\dots,2$ we have
\[
e_{i-1}^{2} \leq C^2_{i} \left(\frac{i-1}{2}\right)^2.
\]
\label{cor:C_bound}
\end{corollary}
Under the assumption that a few consecutive $e_j$'s are at least one half in absolute value, e.g. $|e_k|,\dots,|e_{i}| \geq \frac{1}{2}$, 
the next lemma will give a closed form expression for $C_i$ depending on $e_k,\dots,e_{i}$. As a corollary of the next lemma and 
Proposition \ref{prop:eC}, we will see how $e_k,\dots,e_i$ satisfy a $(k-i+1)$-dimensional ellipsoid equation.

\begin{lemma} Let $n\geq k \geq i \geq 1$. Under the assumption that $|e_k|,\dots,|e_{i}| \geq \frac{1}{2}$ and 
either $k=n$ or $|e_{k+1}|< \frac{1}{2}$ we have
\[ 
C^2_i= \frac{\prod^k_{j=i}\frac{j^2}{4}}{\prod^k_{j=i}\left( \frac{j^2}{4} -\frac{1}{4}\right)}- \sum^k_{j=i+1}\left(e^2_j\frac{\prod^{j-1}_{l=i} \frac{l^2}{4}}{\prod^j_{l=i}\left( \frac{l^2}{4} -\frac{1}{4}  \right)}\right)-e^2_{i}\frac{1}{\frac{i^2}{4} -\frac{1}{4}}.
\]
\end{lemma}
\begin{proof}
We go by reverse induction on $i$.
 The case $i=k$ follows by definition. Assume the result holds for $i+1$. Then
\[
C^2_{i}= A^2_{i}\left(\frac{e_{i}}{C_{i+1}}\right) C_{i+1}^2=  C^2_{i+1} \frac{\frac{i^2}{4}}{\frac{i^2}{4} - \frac{1}{4}}-e^2_{i}\frac{1}{\frac{i^2}{4} - \frac{1}{4}}.
\]
Plugging in $C^2_{i+1}$ immediately gives the result.
\end{proof}
Note that from $e^{(i)2}\leq C^2_{i+1} \mu^{(i)2}$ and the Transference Theorems we obtain $e^2_i\leq \frac{i^2}{4} C^2_{i+1}$.
So under the condition that $|e_k|,\dots,|e_{i+1}| > \frac{1}{2}$ we have that
\begin{equation}
 e^2_i \leq \frac{i^2}{4} \left(\frac{\prod^k_{j=i+1}\frac{j^2}{4}}{\prod^k_{j=i+1}\left( \frac{j^2}{4} -\frac{1}{4}\right)}- \sum^k_{j=i+2}\left(e^2_j\frac{\prod^{j-1}_{l=i+1} \frac{l^2}{4}}{\prod^j_{l=i+1}\left( \frac{l^2}{4} -\frac{1}{4}  \right)}\right) - e^2_{i+1}\frac{1}{\frac{(i+1)^2}{4} -\frac{1}{4}} \right).
\label{equ:e_i}
\end{equation}
The following corollary follows immediately.
\begin{corollary}
If $|e_k|,\dots,|e_{i+1}| > \frac{1}{2}$ for $1 \leq i<k \leq n$, then
\begin{equation}
 e^2_i + \sum^k_{j=i+1}\left(e^2_j\frac{\prod^{j-1}_{l=i} \frac{l^2}{4}}{\prod^j_{l=i+1}\left( \frac{l^2}{4} -\frac{1}{4}  \right)}\right) \leq \frac{\prod^k_{j=i}\frac{j^2}{4}}{\prod^k_{j=i+1}\left( \frac{j^2}{4} -\frac{1}{4}\right)}.
\label{equ:cond}
\end{equation}
\end{corollary}
As in Section \ref{sec:new} we now define a set $S_n$ such that all closest vectors to $t$ are in 
$\{t-\sum^n_{j=1} e_j b^*_j: (e_1,\dots,e_n) \in S_n \}$: 
\[
 S_n:=\left\{(e_1,\dots,e_n)\in T_n:  |e_{i}| \leq \frac{i}{2} \cdot C_{i+1} \right\}.
\]
Note that by Corollary \ref{cor:C_bound}, $S_n$ has the desired property.

\subsection{Bounding the set $S_n$}

We will now bound the number of elements in $S_n$.
Clearly 
\[
S_n \subset S'_n:=\left\{(e_1,\dots,e_n)\in \R^n:  |e_{i}| \leq \frac{i}{2} \cdot C_{i+1},~\ i=1,\dots,n \right\}.
\]
For $k=1,\dots,n$, when $C_{k+1}=1$, define $a_k=\vol\left(S_k'\right)$ and set $a_0=1$.
Clearly $a_{1}=1$ and $a_{n}=\vol(S_n')$.
Further define 
\[
 V_{j,k} := \left\{\begin{array}{ll} 1  \mbox{ for } j=k, \\ 
			\vol\left\{(e_{j+1},\dots,e_{k})\in \R^{k-j}: \mbox{ Equation (\ref{equ:cond}) holds} \right\}  & \mbox{else}. \end{array}\right. 
\]
For a given element $(e_1,\dots,e_k) \in \R^k$ we can
define 
\[
\tau:=\max_{1 \leq i \leq k}\left\{i:  C_{i}=1\right\} = \max_{1 \leq i \leq k}\left\{i: |e_i| \leq \frac{1}{2} \right\},
\] 
allowing to write
\[
 a_{k}=a_{\tau-1}V_{\tau,k}.
\]
We can now partition $S'_n$ into disjunct sets, depending on the possible values of $\tau$, as
\[
 S'_k= \bigcup_{1\leq \tau \leq k} \left\{(e_1,\dots,e_k)\in \R^k:  |e_{i}| \leq \frac{i}{2} \cdot C_{i+1} \mbox{ and } \max_{1 \leq i \leq k}\left\{i: C_{i}=1\right\} =\tau \right\}
\]
So in the case where $C_{k+1}=1$, we have
\begin{equation}
 a_{k}=\sum_{1\leq j \leq k} a_{j-1} V_{j,k}.
\label{equ:a_k}
\end{equation}
In particular
\begin{equation*}
 \vol(S_n')=a_{n}=\sum_{1\leq j \leq n} a_{j-1} V_{j,n}.
\end{equation*}
Using the well known formula for the volume of an ellipsoid, $V_{j,k}$ can be computed (see Appendix \ref{sec:appendix}) as
\begin{equation*}
 V_{j,k} =  \frac{\pi^{(k-j)/2}}{\Gamma\left(\frac{k-j}{2}+1\right)} \frac{k!}{j! 2^{k-j}} \left(\frac{k+1}{j+1}\right)^{1/2} \left(\frac{k}{k+1}\right)^{(k-j)/2}.
\end{equation*}
Clearly
\begin{equation*}
  V_{j,k} \leq \left(\frac{\pi}{4}\right)^{(k-j)/2} \frac{1}{\Gamma\left(\frac{k-j}{2}+1\right)} \frac{k!}{j! } \left(\frac{k+1}{j+1}\right)^{1/2}.
\end{equation*}
Plugging this into Equation (\ref{equ:a_k}), for  $k=1,\dots,n$ we get
\begin{equation*}
 a_{k} \leq \sum_{1\leq j \leq k} a_{j-1} \left(\frac{\pi}{4}\right)^{(k-j)/2} \frac{1}{\Gamma\left(\frac{k-j}{2}+1\right)} \frac{k!}{j! } \left(\frac{k+1}{j+1}\right)^{1/2},
\end{equation*}
which leads to 
\begin{eqnarray*}
  \frac{a_{k}}{\sqrt{k+1} (k+1)!} \left(\frac{4}{\pi}\right)^{k/2} 
& \leq &\frac{1}{k+1} \sum_{1\leq j \leq k} \frac{a_{j-1}}{\sqrt{j+1} j!} \left(\frac{4}{\pi}\right)^{j/2} \frac{1}{\Gamma\left(\frac{k-j}{2}+1\right)}
\\ &\leq&  \frac{1}{k+1} \sum_{1\leq j \leq k} \frac{a_{j-1}}{\sqrt{j} j!} \left(\frac{4}{\pi}\right)^{j/2} \frac{1}{\Gamma\left(\frac{k-j}{2}+1\right)}.
\end{eqnarray*}
So we have a recursively defined upper bound for $a_k$. We will now derive a nicer recursion, the goal to upper bound $a_k$ remains the same however.
Define
\begin{equation*}
\sigma_k:=\frac{a_{k}}{\sqrt{k+1} (k+1)!} \left(\frac{4}{\pi}\right)^{k/2} \mbox{ for } k=0,\dots,n.
\end{equation*}
As $a_0=1$, we get the following recursive relation
\begin{equation*}
\sigma_k \leq \left\{\begin{array}{ll} 1 & \mbox{for } k=0, \\ 
			\frac{1}{k+1} \sum^{k}_{j=1} \frac{\sigma_{j-1}}{\Gamma\left(\frac{k-j}{2}+1\right)}  & \mbox{for } k \geq 1. \end{array}\right. 
\end{equation*}
So setting $s_0:=1$ and
\[
s_k:=\frac{1}{k+1} \sum^{k}_{j=1} \frac{s_{j-1}}{\Gamma\left(\frac{k-j}{2}+1\right)},
\]
 then 
$\sigma_k\leq s_k$ and it is enough to derive an upper bound on $s_k$. 
We can define the following sequence for $n\geq 2$:
\begin{eqnarray}
c_n &:=& \frac{\log s_n }{n\log n }+\frac{1}{2 n } +\frac{(n+2)\log (n+2)}{n \log n} -\frac{1}{\log n} +\frac{\log (\pi/4)}{2 \log n} \nonumber \\ 
& \geq & \frac{\log\left(s_n \sqrt{n}(n+1)!\left(\pi/4\right)^{n/2}\right)}{n\log n } \label{st}\\ 
&  \geq & \frac{\log a_n }{n\log n },\nonumber
\end{eqnarray}
where Eq.~(\ref{st}) is valid because  $(n+1)! \leq e \left(\frac{n+2}{e}\right)^{n+2}$.
\begin{figure}
\centering
\includegraphics[scale=0.8]{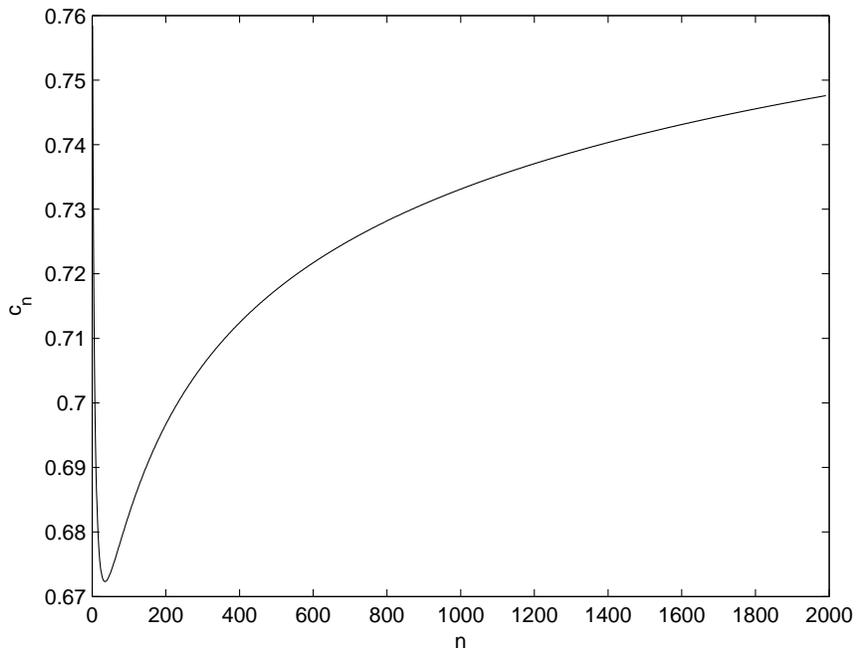}
\caption{The behaviour of $c_n$ for $10 \leq n \leq 2000$.}
\label{fig:c_n}
\end{figure} 
Then
\[
\vol(S_n')=a_{n}\leq n^{c_nn}.
\]
Deriving any useful and provable explicit bound on $s_n$, and therefore on $c_n$, seems to be a nontrivial task. However, numerical computations of $c_n$ suggest that $c_n < 0.75$ for $10< n\leq 2000$ (compare  Figure \ref{fig:c_n}).

\section{Hermite factor}
\label{sec:Hermite}
In this section we will point out the influence of the hermite factor of the dual lattice on the running time of the algorithm.
While the considerations in the previous section give a reduction in the running time for all lattices, this 
section will only give an improvement in the case where the length $\lambda^\times_1$ of a shortest vector 
in the dual lattice satisfies $\lambda^\times_1 \geq \left(\vol(\tL^\times)\right)^{1/n}$.
Let 
\[
 \alpha:=\sqrt{\gamma(\tL^\times)}= \frac{\lambda_1(\tL^\times)}{\left(\vol(\tL^\times)\right)^{1/n}},
\]
denote the hermite factor of the dual lattice $\tL$. 
Now consider the following bound on the length of the error vector $e$
\begin{equation}
\norm{e}^{2}= e^2_1\norm{b^*_1}^2+\dots+e^2_n\norm{b^*_n}^2 \leq \left(\frac{n}{2}\right)^2 \norm{b^*_n}.
\label{equ:volHe}
\end{equation}
Again the number of coefficients satisfying this inequality can be approximated by the volume of the 
ellipsoid:
\begin{multline}
V:=\left|\left\{ (e_1,\dots,e_n)\in \R^n: v=t-\sum^n_{j=1} e_j b^*_j \in \tL \mbox{ and }  \norm{\sum^n_{j=1} e_j b^*_j} \leq \frac{n}{2} \norm{b^*_n} \right\}\right| 
\\ = \frac{\pi^{n/2}}{\Gamma(n/2+1)} \left(\frac{n}{2}\right)^n \frac{ \norm{b^*_n}^n}{\prod^{n}_{j=1} \norm{b^*_j}}.
\end{multline}
Note that if $B=[b_1,\dots,b_n]$ is dual HKZ reduced, then $\norm{b^*_n}= \frac{1}{\lambda_1(\tL^\times)}$ and $\prod^{n}_{j=1} \norm{b^*_j}= \vol(\tL)=\frac{1}{\vol(\tL^\times)}$.
Consequently
\begin{equation}
 \frac{ \norm{b^*_n}^n}{\prod^{n}_{j=1} \norm{b^*_j}} =\left(\frac{\vol(\tL^\times)}{\lambda_1(\tL^\times)}\right)^n= \left(\frac{1}{\alpha}\right)^n.
\end{equation}
So 
\begin{eqnarray}
V & = & \left|\left\{ (e_1,\dots,e_n)\in \R^n: v=t-\sum^n_{j=1} e_j b^*_j \in \tL \mbox{ and }  \norm{\sum^n_{j=1} e_j b^*_j} \leq \frac{n}{2} \norm{b^*_n} \right\}\right| 
\nonumber \\
& = & \left(\frac{\pi}{4}\right)^{n/2}\frac{1}{\Gamma(\frac{n}{2}+1)} \left(\frac{n}{2}\right)^n \left(\frac{1}{\alpha}\right)^n.\label{equ:V}
\end{eqnarray}
While $\alpha$ can be smaller than $1$, the Gaussian heuristic \cite{ng10} suggest that it is bigger than one:
\begin{equation}
\alpha=\frac{\lambda_1(\tL^\times)}{(\vol \tL^\times )^{1/n}} \approx \frac{\Gamma\left(\frac{n}{2}+1\right)^{1/n}}{\sqrt{\pi}}.
\label{equ:GH}
\end{equation}
In fact tests with random integer lattices in the sense of Goldstein and Meier \cite{go03} suggest 
that the heuristic is quite tight for higher dimensions ($>30$). E.g. for dimension $n=30$, the
Gaussian heuristic suggest that $\alpha \approx 1.43$, which is supported by the histogram in Figure \ref{fig:hermite}.
\begin{figure}
\label{fig:hermite}
\centering
\includegraphics[scale=0.43]{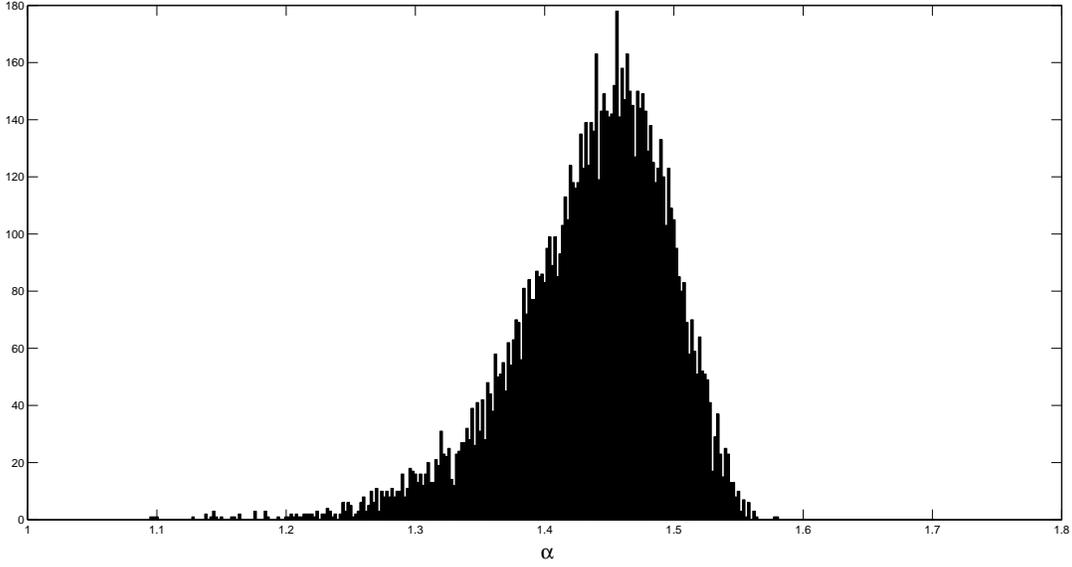}
\caption{Histogram of hermite factors of random lattices with 200-bit prime determinant and dimension $30$.}
\end{figure} 
Assuming that this is in fact the case and plugging
in the Gaussian heuristic into formula (\ref{equ:V}), we obtain
\[
V= \left(\frac{\pi}{4}\right)^{n/2}\frac{1}{\Gamma(\frac{n}{2}+1)} \left(\frac{n}{2}\right)^n \left(\frac{1}{\alpha}\right)^n \approx \left(\frac{\sqrt{n}}{4}\right)^{n}.
\]
So in this case, the average number of points to be enumerated would be
\begin{equation}
 \left(\frac{\sqrt{n}}{4}\right)^{n} = \frac{1}{2^{2n}}n^{n/2}.
\label{equ:ifGH}
\end{equation}

\section{Kannan's algorithm}
\label{sec:Kannan}
In this section we quickly review Kannan's algorithm and the complexity analysis done by Hanrot and Stehl\'e 
\cite{ha07,ha10}.
In contrast to Bl\"omers approach, Kannan's algorithm takes as input a HKZ reduced basis $B$.
Let $e=e_1b^*_1+\dots+e_nb^*_n$ again denote the error vector $v-t$.
Hanrot and Stehl\'e in their analysis use the fact that 
\begin{equation}
 e^2_1\norm{b^*_1}+\dots+e^2_k\norm{b^*_k} \leq \frac{1}{4} \sum^k_{i=1} \norm{b^*_i}^2 \leq \frac{k}{4} \max_{1\leq j\leq k}\norm{b^*_j}^2.
\label{equ:volKa}
\end{equation}
Clearly the volume of the ellipsoid defined by Equation (\ref{equ:volKa}) depends on the lengths of the 
Gram-Schmidt vectors. 
Let us define $C(0):=1$ and for $k \geq 1$
\[
 C(k):= |\{(e_{1},\dots,e_{k}) \in \R^{k}:  \mbox{ Equation (\ref{equ:volKa}) holds}\}|.
\]
We derive a recursive bound for $C(n)$: Let $\tau := \argmax_{1\leq j \leq n}{\norm{b^*_j}}$.
Then we have the following inequality
\begin{equation}
 e^2_{\tau}\norm{b^*_{\tau}}+\dots+e^2_n\norm{b^*_n} \leq e^2_1\norm{b^*_1}+\dots+e^2_n\norm{b^*_n} \leq \frac{n}{4} \norm{b^*_\tau}^2.
\label{equ:ka1}
\end{equation}
As $B$ is HKZ reduced, also the $n-\tau+1$ dimensional lattice $\pi_{\tau}(B)$ is HKZ reduced. Consequently
 by Hermite's \cite{ng10} bound we have that
\[
 \norm{b^*_\tau} \leq \sqrt{\frac{n-\tau+5}{4}} \cdot \vol(\pi_\tau(\tL))^{1/(n-\tau+1)}.
\]
Let us consider 
\[
 C(\tau,n):= |\{(e_{\tau},\dots,e_{n}) \in \R^{n-\tau+1}:  \mbox{ Equation (\ref{equ:ka1}) holds}\}|.
\]
We can compute the volume of $C(\tau,n)$ using the Ellipsoid formula
\begin{eqnarray*}
 C(\tau,n)&=&\frac{\pi^{(n-\tau+1)/2}}{\Gamma(\frac{n-\tau+1}{2}+1)} \left(\frac{n}{4}\right)^{(n-\tau+1)/2} \frac{\norm{b^*_\tau}^{n-\tau+1}}{\prod^n_{j=\tau} \norm{b^*_j}}
 \\&=&\frac{\pi^{(n-\tau+1)/2}}{\Gamma(\frac{n-\tau+1}{2}+1)} \left(\frac{n}{4}\right)^{(n-\tau+1)/2} \frac{\norm{b^*_\tau}^{n-\tau+1}}{\vol(\pi_\tau(\tL))} 
\\ &\leq& \left(\frac{\pi}{4}\right)^{(n-\tau+1)/2}\frac{1}{\Gamma(\frac{n-\tau+1}{2}+1)}  \left(\frac{n-\tau+5}{4}\right)^{(n-\tau+1)/2} n^{(n-\tau+1)/2}.
\end{eqnarray*}
Consequently we get
\[
 C(n) \leq C(\tau-1) C(\tau,n) = C(\tau-1) 2^{c(n-\tau+1)} n^{(n-\tau+1)/2},
\]
for some constant $c$. 
This gives
\[
 C(n) \leq 2^{c'n} n^{n/2},
\]
for some constant $c'$.
As in the previous section we can state the result in the case where the Gaussian heuristic is reached,
i.e. $\norm{b^*_\tau} = \frac{\Gamma\left(\frac{n-\tau+1}{2}+1\right)^{1/(n-\tau+1)}}{\sqrt{\pi}} \cdot \vol(\pi_\tau(\tL))^{1/(n-\tau+1)}$.
This gives 
\[
 C(\tau,n) = \left(\frac{n}{4}\right)^{(n-\tau+1)/2},
\]
and 
\[
 C(n) = \left(\frac{\sqrt{n}}{2}\right)^{n} = \frac{1}{2^{n}}n^{n/2}.
\]

\newpage

\section{Conclusion}
\label{sec:Conclusion}

We have seen that given a dual HKZ-basis, we can solve the closest vector problem 
using the approach by Bl\"omer \cite{bl00} by enumerating 
$n^{c_nn}$ lattice points, with $c_n < 0.75$ for $10< n\leq 2000$. Kannan's algorithm runs faster, as refined analysis thereof implies \cite{ha10}. 
Using Kannan's algorithm, which as input takes a HKZ-basis, it is enough to enumerate 
 $n^{n/2+o(n)}$ lattice points. 
On the other hand we have seen that if the shortest vector of the dual lattice satisfies the Gaussian heuristic,
the transference theorems imply that is enough to enumerate all lattice points 
inside a ellipsoid of volume $\left(\frac{\sqrt{n}}{4}\right)^n$ in order to find the closest vectors.
If the same assumption is made for all gram-schmidt vectors of the HKZ-basis used in Kannan's algorithm, 
the closest lattice points lie inside an ellipsoid of volume $\left(\frac{\sqrt{n}}{2}\right)^n$.
Referring to the case where the Gaussian heuristic as tight as average case, Table \ref{table:1} gives an overview
on the complexities.

\begin{table}[h]
\renewcommand{\arraystretch}{1.3}
\begin{center}
\begin{tabular}{  l l l  l  l l  l } \toprule
Approach	& ~\ ~\ &   original  &~\  ~\ &  refined (worst case) &~\ ~\ & refined (average) \\  \midrule
Kannan	  & ~\ &  	$n^{n+o(n)}$	 &~\ &  $2^{O(n)}n^{n/2}$ &~\ &  $2^{-2n}n^{n/2}$    \\  
Bl\"omer  & ~\ & 	$n!$		 &~\ &  $n^{c_nn}$	&~\ &  $2^{-n}n^{n/2}$  \\ \toprule
\end{tabular}
\end{center}
\caption{Overview on the number of points to enumerate.}
\label{table:1}
\end{table}


\begin{appendix}
\section{Computation of $V_{\tau,k}$ in Section \ref{sec:further}}
\label{sec:appendix}
\begin{eqnarray*}
 V_{\tau,k} &=& \frac{\pi^{(k-\tau)/2}}{\Gamma\left(\frac{k-\tau}{2}+1\right)} \left(\frac{\prod^k_{j=\tau+1}\frac{j}{2}}{\prod^k_{j=\tau+2}\left( \frac{j^2}{4} -\frac{1}{4}\right)^{1/2}}\right)^{k-\tau+1} \prod^k_{j=\tau+2} \frac{ \prod^j_{i=\tau+2}\left( \frac{i^2}{4} -\frac{1}{4}  \right)^{1/2}}{\prod^{j-1}_{i=\tau+1} \frac{i}{2}} 
\\ &=&  \frac{\pi^{(k-\tau)/2}}{\Gamma\left(\frac{k-\tau}{2}+1\right)} \left(\frac{\tau+1}{2}\right)^{k-\tau} \left(\prod^k_{j=\tau+2}\frac{\frac{j}{2}}{\left( \frac{j^2}{4} -\frac{1}{4}\right)^{1/2}}\right)^{k-\tau} \cdot
\\ & & \frac{k!}{(\tau+1)^{k-\tau-1}(\tau+1)!}\prod^k_{j=\tau+2}\prod^{j}_{i=\tau+2} \frac{ \left( \frac{i^2}{4} -\frac{1}{4}  \right)^{1/2}}{ \frac{i}{2}}
\\ &=&  \frac{\pi^{(k-\tau)/2}}{\Gamma\left(\frac{k-\tau}{2}+1\right)} \frac{k!}{\tau! 2^{k-\tau}} \left(\frac{(\tau+2)k}{(\tau+1)(k+1)} \right)^{(k-\tau)/2} \left(\frac{\tau+1}{\tau+2}\right)^{(k-\tau-1)/2} \left(\frac{k+1}{\tau+2}\right)^{1/2}
\\ &=&  \frac{\pi^{(k-\tau)/2}}{\Gamma\left(\frac{k-\tau}{2}+1\right)} \frac{k!}{\tau! 2^{k-\tau}} \left(\frac{k+1}{\tau+1}\right)^{1/2} \left(\frac{k}{k+1}\right)^{(k-\tau)/2}
\\ &\leq& \frac{\pi^{(k-\tau)/2}}{\Gamma\left(\frac{k-\tau}{2}+1\right)} \frac{k!}{\tau! 2^{k-\tau}} \left(\frac{k}{\tau+1}\right)^{1/2} 
\\&=& \left(\frac{\pi}{4}\right)^{(k-\tau)/2} \frac{1}{\Gamma\left(\frac{k-\tau}{2}+1\right)} \frac{k!}{\tau! } \left(\frac{k}{\tau+1}\right)^{1/2} 
\end{eqnarray*}
\end{appendix}

\bibliographystyle{plain}        

\end{document}